\documentclass[12pt]{amsart}
\usepackage{amssymb}
\usepackage{bbm}
\usepackage{mathtools}
\usepackage{enumitem}
\usepackage{tikz-cd}
\usepackage{calc}
\usepackage{caption}

\newcommand{\id}{\mathrm{id}}

\newcommand{\fK}{\Bbbk}

\newcommand{\clcm}[1]{ {}^{#1}\mathcal{M}}
\newcommand{\ot}{\otimes}

\theoremstyle{plain}
\newtheorem{thm}{Theorem}[section]
\newtheorem*{thm*}{Theorem}
\newtheorem*{classification*}{Classification theorem}
\newtheorem{lem}[thm]{Lemma}
\newtheorem{cor}[thm]{Corollary}

\newtheorem*{notation*}{Notation}
\newtheorem{prop}[thm]{Proposition}

\newtheorem{defi}[thm]{Definition}

\theoremstyle{remark}

\newtheorem*{rema*}{Remark}
\newtheorem{exa}[thm]{Example}

\numberwithin{equation}{section}

\begin{document}

\title[Decomposing pointed braided Hopf algebras]{A decomposition Theorem for\\ pointed braided Hopf algebras} 
\author{Istvan Heckenberger}
\address{Fachbereich Mathematik und Informatik,
Philipps-Universit\"at Marburg,
Hans-Meerwein-Str.~6, 35032 Marburg, Germany}
\email{heckenberger@mathematik.uni-marburg.de}

\author{Katharina Sch\"afer}
\address{Fachbereich Mathematik und Informatik,
Philipps-Universit\"at Marburg,
Hans-Meerwein-Str.~6, 35032 Marburg, Germany}
\email{schaef7a@mathematik.uni-marburg.de}

\begin{abstract}
  A known fundamental Theorem for braided pointed Hopf algebras states that for each coideal subalgebra, that fulfils a few properties, there is an associated quotient coalgebra right module such that the braided Hopf algebra can be decomposed into a tensor product of these two. Often one considers braided Hopf algebras in a Yetter-Drinfeld category of an ordinary Hopf algebra. In this case the braided Hopf algebra is in particular a comodule, as well as many interesting coideal subalgebras. We extend the mentioned Theorem by proving that the decomposition is compatible with this comodule structure if the underlying ordinary Hopf algebra is cosemisimple. 
\end{abstract}



\maketitle

\section*{Introduction}

The initial motivation for the discussion in the present paper is twofold. First, in the last years several papers (e.g. \cite{MR3552907}, \cite{MR2179722}, \cite{MR2415067}, \cite{MR3096611}, \cite{MR3413681}) appeared in which
rather general types of one-sided coideal subalgebras of some pointed braided Hopf algebras are studied from different perspectives. Especially the freeness property of Hopf algebras over such coideal subalgebras raises several questions in these contexts.
In a different direction, a characterization of one-sided coideal subalgebras in (braided) Hopf algebra triples is presented in \cite[\S.\,12.4]{MR4164719}, and is used later on essentially to establish the reflection theory of Nichols algebras.

In \cite{MR3552907}, subalgebras of Fomin-Kirillov algebras generated by homogeneous elements with respect to the grading by the symmetric group are studied.
These subalgebras are by construction left coideal subalgebras of the Fomin-Kirillov algebra and are graded by the symmetric group.

Kharchenko deals in \cite{MR2179722} with the free algebra generated by a braided vector space. (If the braiding is rigid, by a result of Takeuchi one may assume that this object is a Yetter-Drinfeld module over a Hopf algebra.) By defining the elements of the generating subspace to be primitive, the free algebra becomes a pointed braided Hopf algebra. Kharchenko proves in this context freeness \emph{of} (not \emph{over}!) "right categorical right coideal subalgebras". After interchanging left and right in his setting, his notion is very close to our
left coideal subalgebras in the category of comodules, see below.

By reading the above mentioned and some similar papers we concluded that there is appreciable interest in the study of (half-categorical) one-sided  coideal subalgebras of pointed braided Hopf algebras.

The second motivation for our work is our attempt to strengthen a very powerful Theorem in \cite[Theorem~6.3.2]{MR4164719} on the decomposition of a pointed braided Hopf algebra into the tensor product of a left coideal subalgebra and a coalgebra right module. For the special case of a connected braided Hopf algebra in ${^{H}_{H}}\mathcal{YD}$ and for Yetter-Drinfeld module left coideal subalgebras $K$, a similar theorem was also proven in \cite[Proposition~3.6]{MR3133699}. We identified the need for a generalization when trying to apply the theorem to classify Nichols algebras over non-abelian groups satisfying a certain property regarding their left coideal subalgebras. For this application it would be essential to choose the above mentioned decomposition such that it is compatible with an additional coaction of a Hopf algebra. We prove in Theorem~\ref{decomp} the following version of the decomposition Theorem: 

\begin{quote}
Let $H$ be a cosemisimple Hopf algebra. Let $A$ be a pointed braided Hopf algebra in ${^{H}_{H}}\mathcal{YD}$ such that $\fK g$ is an $H$-subcomodule and an $H$-submodule for all group-like elements $g\in A$. Let $K$ be a left coideal subalgebra of $A$ in ${^H}\mathcal{M}$ such that $g^{-1}\in K$ for all group-like elements $g\in A\cap K$.
Then there is a left $K$-linear right $A/K^+A$-colinear and left $H$-colinear isomorphism $K\ot A/K^+A\rightarrow A$. 
\end{quote}
Moreover, we also formulate a corollary with a graded version of this claim.
For the proof of Theorem~\ref{decomp}, we follow the steps in Chapter~6.3 in \cite{MR4164719}, but we have to adapt and add arguments in order to get the $H$-colinearity. Hopf modules in ${^A}({^H}\mathcal{M})_K$ play a crucial role here, which we introduce in Section~\ref{sec:Hopfmodules}.
An essential argument in the proof and an interesting statement in itself is that such Hopf modules are free as $K$-modules in the category of left $H$-comodules, what we prove in
Proposition~\ref{pro:freeness}.

\newpage

\section{Preliminaries and examples}

We denote the counit of a coalgebra by $\varepsilon,$ the comultiplication by $\Delta$ and we use Sweedler's notation. If $C$ is a coalgebra and $M$ a left (resp. right) $C$-comodule, we indicate the comodule structure map by ${^C}\delta_M$ (resp.  $\delta^C_M$) and extend Sweedler's notation writing $${^C}\delta_M(m)=m_{(-1)}\ot m_{(0)} \ \ \mathrm{(resp.} \ \ \delta^C_M(m)=m_{(0)}\ot m_{(1)})$$ 
for all $m\in M$. For the set of group-like elements of $C$ we write $G(C)$. If $A$ is an algebra and $N$ a left (resp. right) $A$-module, we denote the module structure map by ${_A}\lambda_N$ (resp. $\lambda_{N,A}$). Sometimes we write $a\cdot n$ (resp. $n\cdot a$) for the action of an element $a\in A$ on $n\in N$.

Let $H$ be a Hopf algebra with bijective antipode over a field $\fK$. Then the category ${_{H}^{H}}\mathcal{YD}$ of Yetter-Drinfeld modules over $H$ is braided monoidal with braiding 
$$c:V\ot W\rightarrow W\ot V,\ v\ot w\mapsto v_{(-1)}\cdot w\ot v_{(0)}.$$    
Let $A$ be a braided Hopf algebra in the category ${_{H}^{H}}\mathcal{YD}$. For details about Hopf algebras in the category of Yetter-Drinfeld modules we refer to Chapters 3 and 4 in \cite{MR4164719}. Let $K$ be a left coideal subalgebra of $A$ in the category ${^H}\mathcal{M},$ that is a left $H$-subcomodule, a subalgebra of $A$ and $\Delta(K)\subseteq A\ot K$.  Let $K^+=\ker(\varepsilon: K\rightarrow \fK)$ and $\overline{A}=A/K^+A$. Then $\overline{A}$ is a quotient coalgebra of $A$ (see Lemma 6.3.3, \cite{MR4164719}) in $\clcm {H}$, since $K$ is a left $H$-subcomodule of $A$ and $\varepsilon$ is a morphism in $\clcm {H}$. The braided Hopf algebra $A$ is a right $\overline{A}$-comodule by $a\mapsto a^{(1)}\ot \overline{a^{(2)}}$. 

As examples of the mentioned objects one can keep the following more or less explicit ones in mind.

\begin{exa}\label{ex:tensoralg}
Let $H=\fK G$ be the group Hopf algebra for a group $G$. Take a Yetter-Drinfeld module $V$ in ${^{H}_{H}}\mathcal{YD}$. Using the isomorphism between $G$-graded vector spaces and $\fK G$-comodules, Yetter-Drinfeld modules $V$ in ${^{H}_{H}}\mathcal{YD}$ are $G$-graded vector spaces $V=\bigoplus_{g\in G}V_g$ with a left $\fK G$-module structure, such that $g\cdot V_h\subseteq V_{ghg^{-1}}$ for all $g,h\in G$. The implied left $\fK G$-comodule structure is given by ${^{\fK G}}\delta_V(v)=g\ot v$ for all $v\in V_g$, $g\in G$. 
Then the tensor algebra $A=T(V)$ is a braided Hopf algebra in ${^{H}_{H}}\mathcal{YD}$ with
$$ \Delta(v)=1\ot v+v\ot 1,\quad
\varepsilon(v)=0,\quad S(v)=-v $$
for all $v\in V$. The $\fK G$-action and $\fK G$-coaction on $A$ are defined to be diagonal, and trivial on $1_A$. If $U$ is a $\fK G$-subcomodule of $V$ or equivalently a subspace generated by $G$-homogeneous elements, then the subalgebra $K=\left\langle U\right\rangle$ generated by $U$ is a left coideal subalgebra of $A$ in $\clcm {H}$. 
\end{exa}

\begin{exa}\label{ex:nicholsalg}
Let $V\in {^{H}_{H}}\mathcal{YD}$. The Nichols algebra of $V$, defined as the quotient $$\mathcal{B}(V)=T(V)/I(V),$$ where $I(V)$ is the largest coideal of $T(V)$ in degree $\geq 2$ is also a braided Hopf algebra in ${^{H}_{H}}\mathcal{YD}$. For detailed informations about Nichols algebras see e.g. Sections~1.6 and 7.1 in \cite{MR4164719}. More generally, one can take any braided Hopf algebra quotient $A$ of $T(V)$ by a coideal two-sided ideal in $\clcm H$. Then subalgebras of $A$ generated by $H$-subcomodules, which are also left $A$-subcomodules via $\Delta$, are left coideal subalgebras of $A$ in $\clcm H$.
\end{exa}

\begin{exa}\label{ex:fominkirillov}
Here we describe a family of quite explicit well-known examples of braided Hopf algebras. Let $n\in \mathbb{N}$ and let $\mathbb{S}_n$ be the symmetric group on $n$ letters. Let $H=\fK \mathbb{S}_n$ be the group Hopf algebra. The set
$$T=\left\{(i\,j)\,|\,1\le i<j\le n\right\}$$
of transpositions in $\mathbb{S}_n$ is a conjugacy class. Let $V$ be the vector space with basis $v_{ij}$, $1\le i<j\le n$, and let $v_{ji}=-v_{ij}$ for all $i\ne j$. Then $V$ is a Yetter-Drinfeld module in ${^{H}_{H}}\mathcal{YD}$, where the left $H$-comodule structure and the left $H$-module structure are given by 
\begin{align*}
{^H}\delta_V(v_{ij})&=(i\,j)\ot v_{ij},\\
{_H}\lambda_V(\sigma\ot v_{ij})&=v_{\sigma(i)\sigma(j)}  
\end{align*}
for all $i\ne j$, $\sigma\in \mathbb{S}_n$. Moreover, $T(V)$ is a braided Hopf algebra by Example~\ref{ex:tensoralg}. The elements
\begin{align*}
  &v_{ij}^2, && 1\le i<j\le n,\\ &v_{ij}v_{jk}+v_{jk}v_{ki}+v_{ki}v_{ij},&& \#\{i,j,k\}=3,\\
  &v_{ij}v_{kl}-v_{kl}v_{ij}, && \#\{i,j,k,l\}=4,
\end{align*}
are primitive in $T(V)$ and the ideal $I$ of $T(V)$ they generate is also a coideal.
The quotient $T(V)/I$ is a braided Hopf algebra called \textbf{Fomin-Kirillov algebra}. The (left coideal) subalgebras generated by $v_{ij}$, $(i\,j)\in G$, where $G$ is any subset of $T$, have been studied in \cite{MR3552907}.
\end{exa}

\section{Freeness of Hopf modules in $\clcm H$}\label{sec:Hopfmodules}
For any object $V$ in the category $({^H}\mathcal{M})_K$ of right $K$-modules left $H$-comodules with left $H$-colinear structure map $\lambda_{V,K},$ the tensor product $A\ot V$ is a right $K$-module in $\clcm {H}$ by the \textbf{(half) braided diagonal action} 
\begin{align*}
A\ot V\ot K &\xrightarrow{\mathrm{id}_{A\ot V}\ot \Delta_K} A\ot V\ot A\ot K \\
            &\xrightarrow{\mathrm{id}_{A}\ot c\ot \mathrm{id}_K} A\ot A\ot V\ot K \\
						&\xrightarrow{\mu\ot \lambda_{V,K}} A\ot V,
\end{align*}
where $\Delta_K$ is the comultiplication of $A$ restricted to $K,$  $\mu$ denotes the multiplication of $A$ and $c: V\ot A\rightarrow A\ot V$ is the linear map defined by 
$$c(v\ot a)= v_{(-1)}\cdot a\ot v_{(0)}$$
for all $v\in V$, $a\in A$. 
This diagonal action can in fact be formulated and used categorically, see e.g. \cite{MR4209967}.

\begin{defi}
 Let $V$ be a left $H$-comodule. Then $V$ is called a \textbf{(left-right) Hopf module in} ${^H}\mathcal{M}$ if $V$ is a left $A$-comodule and a right $K$-module in the category of left $H$-comodules such that the $A$-comodule structure map ${^A}\delta_V: V\rightarrow A\ot V$ is right $K$-linear, where $A\ot V$ is a right $K$-module by the (half) braided diagonal action. 
\end{defi}

In the definition, the $K$-linearity of $^A\delta_V$ is equivalent to the  left $A$-colinearity of the $K$-module structure map $V\ot K\rightarrow V$,
where  $V\ot K$ is a left $A$-comodule by (half) braided diagonal action, that is
$${^A}\delta_{V\ot K}=(\mu_A\ot \id_{V\ot K})(\id_A\ot c\ot \id_A)({^A}\delta_V\ot \Delta_K).$$
We write ${^A}({^H}\mathcal{M})_K$ for the category of left-right Hopf modules in $\clcm H$, where morphisms are left $H$-colinear, left $A$-colinear and right $K$-linear. We will omit the specification ``left-right'' for Hopf modules, since we are not dealing with other Hopf modules here. Note that $A$ is an object in ${^A}({^H}\mathcal{M})_K,$ where the left $A$-comodule structure is given by comultiplication and the right $K$-module structure by restriction of the multiplication in $A$. Moreover, each left coideal subalgebra $$K\subseteq K'\subseteq A$$
in ${^H}\mathcal{M}$ is a subobject of $A$ in ${^A}({^H}\mathcal{M})_K$.

\begin{lem} \label{le:UKhopfmodule}
Let $U$ be a left $A$-comodule in the category $\clcm H$. Then $U\ot K$ is a Hopf module in ${^A}({^H}\mathcal{M})_K$, where the left $H$-coaction is diagonal, the left $A$-coaction is (half) braided diagonal, and the right $K$-action is multiplication on the second tensor factor.   
\end{lem}

\begin{proof}
Since $\clcm {H}$ is a monoidal category with diagonal left $H$-coaction, the tensor product $U\ot K$ is a left $H$-comodule. Multiplication on the second tensor factor with elements of $K$ is a well-defined left $H$-colinear right module structure map on $U\ot K$, since $K$ is a subalgebra and multiplication in $A$ is $H$-colinear.
Moreover, of course the (half) braided diagonal coaction of $A$ is $H$-colinear as a composition of left $H$-colinear maps. By definition, to prove the right $K$-linearity of ${}^A\delta_{U\ot K}$ we have to verify that
\begin{align} \label{eq:Klinear}
{^A}\delta_{U\ot K}\lambda_{U\ot K,K}=\lambda_{A\ot (U\ot K),K}({^A}\delta_{U\ot K}\ot\id_K),
\end{align}
where $\lambda_{A\ot (U\ot K),K}$ is the (half) braided diagonal $K$-action.
For the proof we use the notation
$$ {}^A\delta_V(v)=v^{(-1)}\ot v^{(0)}$$
for all left $A$-comodules $V$ and all $v\in V$.
Then Equation~\eqref{eq:Klinear} holds since
\begin{align*}
  &\lambda_{A\ot (U\ot K),K}({^A}\delta_{U\ot K}\ot\id_K)(u\ot k\ot l)\\
  &=\lambda_{A\ot (U\ot K),K}
  \big((u^{(-1)}(u^{(0)}{}_{(-1)}\cdot k^{(1)}) \ot u^{(0)}{}_{(0)} \ot k^{(2)})
  \ot l\big)\\
  &=u^{(-1)}
  (u^{(0)}{}_{(-2)}\cdot k^{(1)})(u^{(0)}{}_{(-1)}k^{(2)}{}_{(-1)}\cdot l^{(1)})\ot (u^{(0)}{}_{(0)}\ot k^{(2)}{}_{(0)}l^{(2)})\\
  &=u^{(-1)}(u^{(0)}{}_{(-1)}\cdot (kl)^{(1)})\ot u^{(0)}{}_{(0)}\ot (kl)^{(2)}\\
  &={^A}\delta_{U\ot K}(u\ot kl)\\
  &={^A}\delta_{U\ot K}\lambda_{U\ot K,K}(u\ot k\ot l)
\end{align*}
for all $u\in U$, $k,l\in K$. In these equations we used the definitions, the $H$-comodule axiom (in the second equation) and that $A$ is an $H$-comodule algebra  and  $\Delta_A$ is an algebra morphism in ${^{H}_{H}}\mathcal{YD}$ (in the third equation). 
\end{proof}

For the proof of the decomposition Theorem~\ref{decomp} it is an essential step to prove that under some extra assumptions on the objects $H,A$ and $K$, Hopf modules in ${^A}({^H}\mathcal{M})_K$ are free as $K$-modules in the category of left $H$-comodules in the sense of the following definition. 

\begin{defi}\label{free}
Let $C$ be a bialgebra, $B$ an algebra in ${^C}\mathcal{M}$ and $M$ a right $B$-module in ${^C}\mathcal{M}$. Then we say that $M$ is \textbf{right $B$-free in ${^C}\mathcal{M}$} if there is a $C$-subcomodule $N\subseteq M$ such that the restricted $B$-module structure map $N\ot B\rightarrow M$ is bijective.
\end{defi}

Thus, right $B$-free modules in the category of $C$-comodules are free right $B$-modules generated by a $C$-subcomodule. For the trivial bialgebra $C=\fK 1_C$ this is the common definition of free right $B$-modules. Free right $B$-modules (in the category of vector spaces) are known to be projective objects. If we assume that $C$ is  cosemisimple, then $B$-free modules in ${^C}\mathcal{M}$ in the sense of our definition are projective objects in $({^C}\mathcal{M})_B$ by Lemma~\ref{proj} below.
Note additionally, that direct summands of projective objects are projective by a standard argument.

\begin{lem}  \label{proj}
Let $C$ be a cosemisimple bialgebra, $B$ an algebra in ${^C}\mathcal{M}$ and $P$ a $B$-free right $B$-module in ${^C}\mathcal{M}$. Then $P$ is a projective object in $({^C}\mathcal{M})_B$.
\end{lem}

\begin{proof}
Let $M,N$ be objects and $g: M\rightarrow N$ a surjective morphism in $({^C}\mathcal{M})_B$. Let $f: P\rightarrow N$ be a morphism in $({^C}\mathcal{M})_B$. Since $P$ is $B$-free in ${^C}\mathcal{M}$, $P$ is isomorphic to $Q\ot B$ for a $C$-subcomodule $Q$ of $P$. Since ${^C}\mathcal{M}$ is semisimple, there exists a $C$-subcomodule $M'\subseteq M$ such that
$$\big(\operatorname{Ker}(g)\cap g^{-1}(f(Q))\big)\oplus M'=g^{-1}(f(Q)).$$ 
Thus the restricted map $g:M'\rightarrow f(Q)$ is an isomorphism of $C$-comodules. Let $h: f(Q)\rightarrow M'$ be its inverse and $\varphi$ the $B$-linear extension of the $C$-colinear map 
$$hf: Q\rightarrow M'.$$ 
Then $\varphi:P\rightarrow M$ is left $C$-colinear, right $B$-linear and $f=g\varphi$.
\end{proof}

We are going to prove that Hopf modules in ${^A}({^H}\mathcal{M})_K$ are $K$-free in ${^H}\mathcal{M}$ in the sense of Definition \ref{free} under some extra assumptions on the Hopf algebra $H,$ the braided Hopf algebra $A\in {^{H}_{H}}\mathcal{YD}$ and the left coideal subalgebra $K\subseteq A$. A (braided) Hopf algebra is called pointed if it is pointed as a coalgebra, which means that all its simple subcoalgebras are one-dimensional.  

Recall that $A$ is a Hopf algebra in
${_{H}^{H}}\mathcal{YD}$.

\begin{lem} \label{le:onedimsub}
  Let $U\subset A$ be a one-dimensional subspace which is both an $H$-subcomodule and a left $A$-subcomodule of $A$. Then $U=\Bbbk g$ for some $g\in G(A)$ with ${}^H\delta(g)=1\ot g$.
\end{lem}

\begin{proof}
  Since $\Delta(U)\subseteq A\ot U$, and since $U\ne 0$ and $(\id \ot \varepsilon)\Delta (x)=x$ for all $x\in A$, there exists $g\in U$ with $\varepsilon (g)=1$. Moreover, $\Delta(g)=g'\ot g$ for some $g'\in A$ because of $\dim U=1$ and $\Delta(U)\subseteq A\ot U$. Equation $(\id \ot \varepsilon)\Delta (g)=g$ implies that $g'=g$ and hence $g\in G(A)$.
  
By assumption, there exists $h\in H$ with
  $$ {}^H\delta_A(g)=h\ot g.$$
  Since $\Delta $ is a morphism in $\clcm H$, it follows that
  $$ h\ot g\ot g
  =h\ot \Delta(g)
  ={}^H\delta_{A\ot A}(\Delta(g))=h^2\ot g\ot g. $$
  Hence $h^2=h$. Since $h\in G(H)$ due to the comodule axiom for ${}^H\delta_A$, and since $H$ is a Hopf algebra, it follows that $h$ is invertible and hence $h=1$.
\end{proof}

\begin{lem}
\label{le:gsubcomod}
Assume that $A$ is pointed and that each group-like element of $A$ spans an $H$-subcomodule of $A$. Then
\begin{enumerate}
    \item $^H\delta(g)=1\ot g$ for all $g\in G(A)$,
    \item the multiplication of $A$ induces a group structure on $G(A)$, and
    \item
for each non-zero object $V\in {}^A(\clcm H)$ there exist $g\in G(A)$ and a subobject $U\ne 0$ of $V$ in $\clcm H$ with ${}^A\delta_V(u)=g\ot u$ for all $u\in U$.
\end{enumerate}
\end{lem}

\begin{proof}
  (1) follows directly from Lemma~\ref{le:onedimsub}.
  
  (2) Let $g,h\in G(A)$. Then
  $\Delta(gh)=\Delta(g)\Delta(h)=gh\ot gh$ because of (1).
  Moreover, $g$ is invertible with inverse $S(g)$ because of the Hopf algebra axiom. Finally, again by (1),
  $$ 1\ot 1=\Delta(gg^{-1})
  =\Delta(g)\Delta(g^{-1})
  =g\,(g^{-1})^{(1)}\ot g\,(g^{-1})^{(2)},
  $$
  which implies that $g^{-1}\in G(A)$.
  
  (3) Since $V\ne 0$ and $A$ is pointed, there exist $0\ne v\in V$ and $g\in G(A)$ with $^A\delta_V(v)=g\ot v$.
  Let
  $$ U=\left\{f(v_{(-1)})v_{(0)}\mid f\in H^*\right\}$$
  be the $H$-subcomodule of $V$
  generated by $v$, where $^H\delta_V(v)=v_{(-1)}\ot v_{(0)}$.
  Note that ${}^H\delta_A(g)=1\ot g$ by Lemma~\ref{le:onedimsub}, and hence $H$-colinearity
  of $^A\delta_V$ implies that
  \begin{align*}
   v_{(-1)}\ot {}^A\delta_V(v_{(0)})&=(\id \ot {}^A\delta _V){}^H\delta_V(v)\\
   &={}^H\delta_{A\ot V} {}^A\delta _V(v)\\ &={}^H\delta_{A\ot V}(g\ot v)=v_{(-1)}\ot g \ot v_{(0)}.
  \end{align*}
  Therefore $^A\delta_V(u)=g\ot u$ for all $u\in U$, which proves the lemma.
  \end{proof}

\begin{lem}
\label{le:VKsubcomod}
Assume that $A$ is pointed and that for each $h\in G(A)$, $\Bbbk h$ is an $H$-submodule of $A$. Let $g\in G(A)$ and let $0\ne V\in {}^A(\clcm H)$ such that
$^A\delta _V(v)=g\ot v$ for all $v\in V$. Then the simple $A$-subcomodules of $V\ot K$ are spanned as a vector space by
$v\ot g^{-1}f$ for some $f\in G(A)$ and $0\ne v\in V$.
\end{lem}

\begin{proof}
Let $W$ be a simple $A$-subcomodule of $V\ot K$. Since $A$ is pointed, $W$ is one-dimensional. Let $f\in G(A)$ and $0\ne x\in W$ with
$$ {^A}\delta_{V\ot K}(x)=f\ot x. $$ We write $x=\sum_i v_i\ot k_i$ with suitable linearly independent elements $v_i \in V$ and elements $0\ne k_i\in K$. Then on the one hand 
$${^A}\delta_{V\ot K}\left(\sum_i v_i\ot k_i\right)={^A}\delta_{V\ot K}(x)=f\ot \sum_i v_i\ot k_i,$$
and on the other hand 
$${^A}\delta_{V\ot K}\left(\sum_i v_i\ot k_i\right)=\sum_i g(v_{i(-1)}\cdot k_i^{(1)})\ot v_{i(0)}\ot k_i^{(2)}.$$
Applying to both expressions $\id \ot \id \ot \varepsilon$, we obtain that
$$ f\ot \sum_i \varepsilon(k_i)v_i
=\sum_i g(v_{i(-1)}\cdot k_i)
\ot v_{i(0)} \qquad (\text{in $A\ot V$}) $$
or, since $g$ is invertible in $A$ (with inverse $S(g)$),
$$ g^{-1}f\ot \sum_i \varepsilon(k_i)v_i
=\sum_i (v_{i(-1)}\cdot k_i)
\ot v_{i(0)}. $$
Next we apply
$(\id \ot S_H^{-1}\ot \id)(\id \ot {}^H\delta_V)$, flip the first two tensor factors, and act with the first tensor factor (in $H$) on the second. This results in the equation
$$ \sum_i S_H^{-1}(v_{i(-1)})\cdot (g^{-1}f)
\ot \varepsilon (k_i)v_{i(0)}
=\sum_i\big(S_H^{-1}(v_{i(-1)})v_{i(-2)}\cdot k_i\big)\ot v_{i(0)}.
$$
Recall that $A$ is an $H$-module algebra and $g,f\in G(A)$. Moreover, the antipode $S$ of $A$ is $H$-linear.
Thus $\Bbbk g^{-1}f=\Bbbk S(g)f$ is an $H$-submodule of $A$ by the assumption on group-like elements, and
it follows that
$$ \sum_i k_i \ot v_i\in 
\Bbbk g^{-1}f \ot V. $$
This implies that $x=\sum_iv_i\ot k_i\in V\ot g^{-1}f$.
\end{proof}

The following Proposition is a generalization of Proposition 6.3.6 in \cite{MR4164719}.

\begin{prop} \label{pro:freeness} 
Assume that the Hopf algebra $H$ is cosemisimple. 
\begin{enumerate}
\item Assume that any non-zero Hopf module $V$ in ${^{A}}({^H}\mathcal{M})_K$ 
contains a non-zero Hopf submodule which is $K$-free in $\clcm H$.
Then any Hopf module in ${^{A}}(\clcm H)_K$ is $K$-free in $\clcm H$.
\item Assume that $A$ is pointed and that
\begin{enumerate}
\item $g^{-1}\in K$ for all $g\in G(A)\cap K$,
\item for all $g\in G(A)$, the subspace $\fK g\subseteq A$ is a left $H$-subco\-module (i.e. ${^H}\delta_A(g)=1 \ot g$),
\item for all $g\in G(A),$ $\fK g$ is a left $H$-submodule. 
\end{enumerate}
Then any Hopf module in  ${^{A}}({^H}\mathcal{M})_K$ 
is $K$-free in ${^H}\mathcal{M}$.
\end{enumerate}  
\end{prop}

\begin{proof}
(1) Let $V$ be a non-zero Hopf module in ${^{A}}({^H}\mathcal{M})_K$ with $K$-module structure map $\lambda_{V,K}$. Let $T$ be the set of left $H$-subcomodules $X$ of $V$ such that $\lambda_{V,K}(X\ot K)\subseteq V$ is an object in ${^{A}}({^H}\mathcal{M})_K$ and $\lambda_{V,K}$ restricted to $X\ot K$ is a monomorphism. The set $T$ contains $0$ and is partially ordered by inclusion. The union of the elements of a totally ordered subset of $T$ is an element in $T$. Hence by Zorn's Lemma there is a maximal element $\widetilde{X}$ of $T$. Let $U=\lambda_{V,K}(\widetilde{X}\ot K)$ and assume for a contradiction that $U\neq V$. Then $V/U$ is a non-zero object in ${^{A}}({^H}\mathcal{M})_K$. By assumption, there is a non-zero Hopf submodule $V'$ of $V$ strictly containing $U$ such that $V'/U$ is $K$-free in $\clcm H$. Thus there is an $H$-subcomodule $X'$ of $V'$ with $\widetilde{X}\subset U\subsetneq X'$ such that the restricted natural induced $K$-module structure map $X'/U \ot K\rightarrow V'/U$ is bijective. Since ${^H}\mathcal{M}$ is semisimple, there is an $H$-subcomodule $Y'\neq 0$ such that $X'=U\oplus Y'$.
Let $Y=Y'+\widetilde{X}$. Then
$$X'/U=(U\oplus Y')/U=(U+Y)/U\cong Y/\widetilde{X}$$
as $H$-comodules, since $\widetilde{X}=Y\cap U$.
Thus we have the following natural exact sequences with isomorphisms between the kernels and the cokernels  
\begin{center}
\begin{tikzcd} 
0 \arrow[r] & \widetilde{X}\ot K \arrow[d,"\cong"]\arrow[r] & Y\ot K \arrow[d]\arrow[r]    &   Y/\widetilde{X}\ot K \arrow[d,"\cong"]  \arrow[r] & 0 \\
0 \arrow[r] & U \arrow[r] & V' \arrow[r]    &   V'/U    \arrow[r] & 0 .	
\end{tikzcd}
\end{center}
Hence $Y\ot K$ is isomorphic to the Hopf submodule $V'\subseteq V$. This is a contradiction to the maximality of $\widetilde{X}$ and thus $V$ is $K$-free in ${^H}\mathcal{M}$.

(2) Note that the claim in brackets behind assumption (b) holds by Lemma~\ref{le:gsubcomod}.

Let $V$ be a non-zero Hopf module in ${^{A}}({^H}\mathcal{M})_K$. By (1), it suffices to prove that $V$ contains a non-zero Hopf submodule which is $K$-free in ${^H}\mathcal{M}$. By Lemma~\ref{le:gsubcomod}, for which we use assumption (b) and that $A$ is pointed, there exist $g\in G(A)$ and a non-zero $H$-subcomodule $U$ of $V$ with
$$ {}^A\delta_V(u)=g\ot u \quad \text{for all $u\in U$.}
$$
By Lemma \ref{le:UKhopfmodule}, the tensor product $U\ot K$ is an object in ${^{A}}({^H}\mathcal{M})_K$, where the left $H$-coaction and the left $A$-coaction are diagonal, and the right $K$-action is multiplication on the second tensor factor.
Let
$$ \varphi: U\ot K\to V$$
be the $K$-module structure map $\lambda_{V,K}$
restricted to $U\ot K$. Then $\varphi$ is a morphism
in ${^{A}}({^H}\mathcal{M})_K$ and we are going to
prove that it is a monomorphism.
Now assume for a contradiction that $\ker \varphi$ is not trivial. Note that $\ker \varphi $ is a subobject of $U\ot K$ in $^A(\clcm H)$. In particular, since $A$ is pointed,
$\ker \varphi\subset U\ot K$ contains a simple $A$-subcomodule. By assumption (2)(c) we may apply Lemma~\ref{le:VKsubcomod} to $U\ot K$. It follows that $u\ot g^{-1}f\in \ker \varphi$ for some $0\ne u\in U$ and $f\in G(A)$. Moreover, $g^{-1}f\in G(A)\cap K$ by Lemma~\ref{le:gsubcomod}(2). Since $g^{-1}f$ is invertible in $K$ by assumption (a) and since $\ker \varphi$ is a $K$-submodule, it follows that $u\ot 1 \in \ker \varphi$, a contradiction to $u\ne 0$ in $V$. This completes the proof of the proposition.
\end{proof}

\section{The $H$-coaction compatible version of the decomposition Theorem}

 An important part of the proof of Theorem~\ref{decomp} below is to prove that $A$ is an injective object in $({^H}\mathcal{M})^{\overline{A}}$. For this we need some preliminary lemmata. In the next lemma we consider $A\ot_K A$ as a right $\overline{A}$-comodule in ${^H}\mathcal{M}$ with diagonal left $H$-coaction and with right $\overline{A}$-comodule structure map
$$\delta^{\overline{A}}_{A\ot_K A}: A\ot_K A \rightarrow A\ot_K A\ot \overline{A}, \quad x\ot y \mapsto x\ot y^{(1)}\ot \overline{y^{(2)}}.$$ 
This is well-defined, since $A\rightarrow A\ot \overline{A}$,  $y \mapsto y^{(1)}\ot \overline{y^{(2)}}$ is left $K$-linear, where $ A\ot \overline{A}$ is a left $K$-module by multiplication on the first tensor factor. Indeed, the $K$-linearity can be verified by the following calculation, where $\pi:A\rightarrow \overline{A}$ denotes the canonical map.  
\begin{align*}
(\id\ot\pi)\Delta\mu(k\ot a)&=k^{(1)}(k^{(2)}{}_{(-1)}\cdot a^{(1)})\ot \pi(k^{(2)}{}_{(0)}a^{(2)})\\
                            &=k^{(1)}(\varepsilon(k^{(2)}{}_{(0)})k^{(2)}{}_{(-1)}\cdot a^{(1)})\ot\pi(a^{(2)})\\
                            &=ka^{(1)}\ot\pi(a^{(2)})\\
                            &=(\mu\ot\id_{\overline{A}})(\id_K\ot(\id_A\ot\pi)\Delta)(k\ot a)
\end{align*}
for all $k\in K$ and $a\in A$, where $\mu$ is the multiplication in $A$. The second equation holds, since $K$ is a left coideal, a left $H$-subcomodule and $\pi(ka)=\varepsilon(k)\pi(a)$ for all $k\in K$, $a\in A$. 
Moreover, $\delta^{\overline{A}}_{A\ot_K A}$ is left $H$-colinear.

\begin{lem} \label{can}
The canonical map 
$$\operatorname{can}: A\ot_K A\rightarrow A\ot\overline{A}, \quad x\ot y\mapsto xy^{(1)}\ot \overline{y^{(2)}}$$
\noindent
is a right $\overline{A}$-colinear left $H$-colinear bijective map.
\end{lem}

\begin{proof}
The $H$- and $\overline{A}$-colinearity properties follow directly from the remarks above the lemma.
For the bijectivity of $\operatorname{can}$ we follow the proof in Lemma~6.3.4 in \cite{MR4164719}. More precisely, we write both $A\ot_KA$ and $A\ot \overline{A}$ as quotients of $A\ot A$ and provide an automorphism of $A\ot A$ which gives rise to the automorphism $\operatorname{can}$ and to its inverse on the quotients.

Recall the exact sequences
\begin{center}
\begin{tikzcd} 
  A\ot K\ot A \arrow[rr,"\mathrm{id}\ot \mu-\mu \ot \mathrm{id}"] && A\ot A \arrow[r] & A\ot _KA \arrow[r] & 0,
\end{tikzcd}
\begin{tikzcd} 
  A\ot K\ot A \arrow[rr,"\mathrm{id}\ot \mu-\mathrm{id}\ot \varepsilon \ot \mathrm{id}"] && A\ot A \arrow[r] & A\ot \overline{A} \arrow[r] & 0.
\end{tikzcd}
\end{center}
In view of the definition of $\operatorname{can}$ there is a natural choice for the automorphism of $A\ot A$. For convenience we work with a slightly more general definition:

For any right $A$-module left $H$-comodule $X$ (e.g. $X=A$) let 
$$\Phi_X: X\ot A\rightarrow X\ot A, \qquad x\ot a \mapsto xa^{(1)}\ot a^{(2)}.$$
The morphisms $\Phi_X$ are $H$-colinear and bijective with inverse
$$ \Phi_X^{-1}:X\ot A\rightarrow X\ot A,
\qquad x\ot a \mapsto xS(a^{(1)})\ot a^{(2)}.$$
We use the morphisms $\Phi_X$ to construct an isomorphism between the above two exact sequences. We are almost done:
\begin{center}
\begin{tikzcd} 
A\ot K\ot A \arrow[d,"?"]\arrow{r}{\mathrm{id}_A\ot\mu-\mu\ot\mathrm{id}_A} &[5em] A\ot A \arrow[d,"\Phi_A"]\arrow[r] & A\ot_K A \arrow[d,dashrightarrow,"\operatorname{can}"]\arrow[r] & 0 &\\
A\ot K\ot A \arrow[r,"\mathrm{id}_A\ot\mu-\mathrm{id}_A\ot\varepsilon\ot\mathrm{id}_A"] & A\ot A \arrow[r] & A\ot \overline{A}   \arrow[r] & 0.&		\end{tikzcd}
\end{center}
The isomorphism corresponding to ? in the latter diagram can be given by composing two maps of the form $\Phi_X$. Indeed, the restrictions of $\Phi_A$ and $\Phi_A^{-1}$ to $A\ot K$ induce bijections 
$$\Phi, \Phi^{-1}: A\ot K \rightarrow A\ot K.$$
Thus $\Phi\ot \id_A$ is an automorphism of $A\ot K\ot A$. Moreover, $\Phi_{A\ot K}$ is an automorphism of $A\ot K\ot A$, where $A\ot K$ is a right $A$-module via
$$ A\ot K\ot A \overset{\id \ot c}{\longrightarrow }A\ot A\ot K \overset{\mu }{\to }A\ot K, $$
and $c$ is the braiding.
(The definition works since $K$ is a left $H$-subcomodule of $A$, and hence $c(K\ot A)\subseteq A\ot K$.) Then clearly
$$ \Psi =\Phi_{A\ot K} (\Phi \ot \id_A )$$
is an isomorphism of $A\ot K\ot A$.
It is a routine calculation that the left square in the above diagram commutes when ? is replaced by $\Psi$. This completes the proof of the lemma.
\end{proof}

\begin{lem} \label{Hombij}
Let $D$ be a bialgebra, $C$ a coalgebra in ${^D}\mathcal{M}$ and $X$ a left $D$-comodule. For any right $C$-comodule $V$ in ${^D}\mathcal{M}$, the map
$$ {^D}\mathrm{Hom}^{C}(V,X\ot C)\stackrel{\cong}{\longrightarrow} {^D}\mathrm{Hom}(V,X), \quad f\mapsto (\mathrm{id}\ot\varepsilon) f $$
is bijective with inverse given by $\varphi\mapsto(\varphi\ot\mathrm{id})\delta^C_V$, where $X\ot C$ is a right $C$-comodule with structure map $\mathrm{id}_X\ot \Delta$ and a left $D$-comodule by diagonal coaction.
\end{lem}

\begin{proof}
For the category of vector spaces instead of the category of left $D$-comodules there is a proof of the claim in Lemma~1.2.10 in \cite{MR4164719}, which we essentially reproduce here. Let $f\in {^D}\mathrm{Hom}^{C}(V,X\ot C)$. Then $(\mathrm{id}\ot\varepsilon) f$ is left $D$-colinear, since $\varepsilon$ and $f$ are. We have
\begin{align*}
((\mathrm{id}_X\ot \varepsilon)f\ot \mathrm{id}_C)\delta^C_V(v)&=(\mathrm{id}_X\ot \varepsilon\ot \mathrm{id}_C)(\mathrm{id}_X\ot \Delta)f(v)\\
  &=(\mathrm{id}_X\ot \mathrm{id}_C)f(v)=f(v)
\end{align*}
for all $v\in V,$ where we first use right $C$-colinearity of $f$ and then the counit axiom. Conversely, let $\varphi\in {^D}\mathrm{Hom}(V,X)$. Then $(\varphi\ot\mathrm{id})\delta^C_V$ is right $C$-colinear because of the coassociativity of $\delta^C_V$ and left $D$-colinear, since $\varphi$ and $\delta^C_V$ are. Moreover,
$$
(\mathrm{id}_X\ot \varepsilon)(\varphi\ot\mathrm{id}_C)\delta^C_V(v)=\varphi(v_{(0)})\varepsilon(v_{(1)})=\varphi(v)
$$
for all $v\in V$.
\end{proof}

\begin{lem} \label{directinj}
Let $D$ be a cosemisimple bialgebra, $C$ a coalgebra in ${^D}\mathcal{M}$ and $V$ a right $C$-comodule in ${^D}\mathcal{M}$. If there is a left $D$-comodule $X$ such that $V$ is a direct summand of $X\ot C$ as left $D$-subcomodule right $C$-subcomodule, then $V$ is an injective object in $({^D}\mathcal{M})^C$. 
\end{lem}

\begin{proof}
For the category of vector spaces instead of the category of left $D$-comodules there is a proof of the claim in Proposition~6.3.8 in \cite{MR4164719}. Since direct summands of an injective object in $({^D}\mathcal{M})^C$ are injective, it suffices to show that $X\ot C$ is injective for any left $D$-comodule $X$. Let $U,W\in ({^D}\mathcal{M})^C$ and $i:U\rightarrow W$ a monomorphism in $({^D}\mathcal{M})^C$ and $f:U\rightarrow X\ot C$ a morphism in $({^D}\mathcal{M})^C$. From 
Lemma~\ref{Hombij} we know that there is a left $D$-colinear map $g: U\rightarrow X$ such that $$f(u)=(g\ot\mathrm{id}_C)\delta^C_U(u)$$ for all $u\in U$. Since ${^D}\mathcal{M}$ is semisimple, $i(U)$ has a $D$-comodule complement in $W$ and thus there is a $D$-colinear map $g_1:W\rightarrow X$ such that $g=g_1i$. Let 
$$g_2:W\rightarrow X\ot C,\quad w\mapsto g_1(w_{(0)})\ot w_{(1)}.$$
Then $g_2$ is right $C$-colinear and 
\begin{align*}
g_2i(u)&=g_1(i(u)_{(0)})\ot i(u)_{(1)}\\
       &= g_1(i(u_{(0)}))\ot u_{(1)}\\
       &=g(u_{(0)})\ot u_{(1)}\\
       &=f(u)
\end{align*}
for all $u\in U$. Moreover, $g_2$ is $D$-colinear.  
\end{proof}

Now all preparations are made to prove the following main Theorem. There the $K$-module structure on $K\ot \overline{A}$ is defined by multiplication on the first tensor factor and the right $\overline{A}$-comodule structure by comultiplication on the second tensor factor.

\begin{thm} \label{decomp}
Let $H$ be a cosemisimple Hopf algebra. Let $A$ be a pointed braided Hopf algebra in ${^{H}_{H}}\mathcal{YD}$ such that $\fK g$ is an $H$-subcomodule and an $H$-submodule for all $g\in G(A)$. Let $K$ be a left coideal subalgebra of $A$ in ${^H}\mathcal{M}$ such that $g^{-1}\in K$ for all $g\in G(A)\cap K$. Let $\overline{A}=A/K^+A$.
Then the following hold.
\begin{enumerate}
\item $A$ is an injective object in $({^H}\mathcal{M})^{\overline{A}}$. 
\item There is a left $K$-linear right $\overline{A}$-colinear and left $H$-colinear isomorphism $K\ot \overline{A}\rightarrow A$. 
\end{enumerate}
\end{thm}

\begin{proof}
(1) Since $A$ and $K$ and thus $A/K$ are Hopf modules in ${^A}({^H}\mathcal{M})_K,$ it follows from Proposition~\ref{pro:freeness} that $A/K$ is $K$-free in ${^H}\mathcal{M}$. Then from Lemma \ref{proj} we know that $A/K$ is projective in $({^H}\mathcal{M})_K$ and thus the natural exact sequence
\begin{center}
\begin{tikzcd} 
  0 \arrow[r] & K \arrow[r] & A \arrow[r]    &   A/K    \arrow[r] & 0 	            
\end{tikzcd}  
\end{center}
splits and $K$ is a direct summand of $A$ as right $K$-module left $H$-comodule. We can decompose 
$$A\ot_K A=(K\oplus A')\ot_K A=K\ot_K A\oplus A'\ot_K A,$$
where $A'$ is a subobject of $A$ in $({^H}\mathcal{M})_K$. Both of the summands are $\overline{A}$-comodules, because the $\overline{A}$-comodule structure is defined on the second tensor factor. Using the isomorphism $\mathrm{can}$ in Lemma \ref{can} it follows that $A\cong K\ot_K A$ is a direct summand of $A\ot \overline{A}$ in the category $({^H}\mathcal{M})^{\overline{A}}$, where $A$ is embedded into $A\otimes \overline{A}$ via
$$ a\mapsto a^{(1)}\ot \overline{a^{(2)}}.
$$
Then from Lemma~\ref{directinj} we conclude that $A$ is an injective object in  $({^H}\mathcal{M})^{\overline{A}}$.

(2) Since $A$ is pointed, the map $G(A)\rightarrow G(\overline{A})$, $g\mapsto \overline{g}$ is surjective (\cite[Proposition~5.4.2]{MR4164719}). Choose a map $\gamma:G(\overline{A})\rightarrow G(A)$ with $\overline{\gamma(\overline{g})}=\overline{g}$ for all $\overline{g}\in G(\overline{A})$. Then the linear map $f:\fK G(\overline{A})\rightarrow A$, $\overline{g}\mapsto \gamma(\overline{g})$ is right $\overline{A}$-colinear and left $H$-colinear, since by assumption $\fK g$ is an $H$-subcomodule for all $g\in G(A)$ and thus $^H\delta(g)=1\ot g$ for all $g\in G(A)$ (see Lemma~\ref{le:gsubcomod}). Since $A$ is an injective object in  $({^H}\mathcal{M})^{\overline{A}}$ by (1), the map $f$ can be extended to a right $\overline{A}$-colinear left $H$-colinear map $h:\overline{A}\rightarrow A$. Define the linear map
\begin{align*}
\varphi: K\otimes \overline{A}\rightarrow A,\quad k\otimes \overline{a}\mapsto kh(\overline{a})
\end{align*}
for all $k\in K$, $\overline{a}\in \overline{A}$. Obviously, the map $\varphi$ is left $K$-linear, right $\overline{A}$-colinear and left $H$-colinear. We want to prove first that $\varphi$ is a monomorphism. Assume, for sake of contradiction, that $\ker\varphi$ is non-trivial. Then it contains a simple $\overline{A}$-subcomodule. Since $\overline{A}$ is pointed, the simple $\overline{A}$-subcomodules are of the form $\fK k\otimes \overline{g}$, where $k\in K$, $k\neq 0$ and $\overline{g} \in G(\overline{A})$. But $$\varphi(k\ot\overline{g})=kh(\overline{g})\neq 0,$$ since $h(\overline{g})$ is a group-like element of $A$ and thus invertible. Hence $\varphi $ is a monomorphism. The right $\overline{A}$-comodule $K\otimes \overline{A}$ is injective (Lemma \ref{directinj}) and thus the right $\overline{A}$-comodule monomorphism $\varphi$ is bijective.
\end{proof}

Note that the assumptions on the braided Hopf algebra $A$ in the last Theorem are fulfilled in particular if $A$ is an $\mathbb{N}_0$-graded connected Hopf algebra. For example the braided Hopf algebras described in Examples~\ref{ex:tensoralg}-\ref{ex:fominkirillov} are connected.  

A direct consequence of the decomposition Theorem \ref{decomp} is the following $\mathbb{N}_0$-graded version. We equip $H$ with the trivial $\mathbb{N}_0$-grading, that is $H(0)=H$ and $H(n)=0$ for all $n>0$. An $\mathbb{N}_0$-graded Hopf algebra in ${^{H}_{H}}\mathcal{YD}$ is a Hopf algebra in ${^{H}_{H}}\mathcal{YD}(\mathbb{N}_0\mathrm{-Gr}\mathcal{M}_{\fK})$. For informations about graded structures and especially about graded Yetter-Drinfeld modules we refer to Chapters 5.1 and 5.5 in \cite{MR4164719}.

\begin{cor}
Let $H$ be a cosemisimple Hopf algebra. Let $A$ be an $\mathbb{N}_0$-graded pointed braided Hopf algebra in ${^{H}_{H}}\mathcal{YD}$ such that $\fK g$ is an $H$-subcomodule and an $H$-submodule for all $g\in G(A)$. Let $K$ be an $\mathbb{N}_0$-graded left coideal subalgebra of $A$ in ${^H}\mathcal{M}$ such that $g^{-1}\in K$ for all $g\in G(A)\cap K$. Let $\overline{A}=A/K^+A$.
Then there is a left $K$-linear right $\overline{A}$-colinear left $H$-colinear and $\mathbb{N}_0$-graded isomorphism $K\ot \overline{A}\rightarrow A$.  
\end{cor} 

\begin{proof}
We extend the $\mathbb{N}_0$-gradings to $\mathbb{Z}$-gradings setting the homogeneous parts equal to zero for each negative integer. Let $H'$ be the Hopf algebra $H\ot\fK\mathbb{Z}$, where $\fK\mathbb{Z}=\Bbbk[z,z^{-1}]$ denotes the group Hopf algebra with group-like generators $z,z^{-1}$. Then $A$ is a left $H'$-comodule by $${^{H'}}\delta_A(a)=a_{(-1)}\ot z^n \ot a_{(0)}$$ for all homogeneous $a\in A(n),$ where ${^H}\delta_A(a)=a_{(-1)} \ot a_{(0)}$. Of course, $A$ is also a left $H'$-module via $${_{H'}}\lambda_A(h\ot x\ot a)={_H}\lambda_A(h\ot a)$$ for all $h\in H$, $x\in\fK\mathbb{Z}$, $a\in A$. Then $A$ (with its initial Hopf algebra structure maps) is a pointed Hopf algebra in ${^{H'}_{H'}}\mathcal{YD}$ and as such it fulfills the assumptions of Theorem \ref{decomp}, since they can be transferred from the assumptions in the present corollary. Since $K$ is a $\mathbb{Z}$-graded left $H$-comodule, it is also a subcomodule of the left $H'$-comodule $A$. Then Theorem \ref{decomp} states that there is a left $K$-linear right $\overline{A}$-colinear left $H'$-colinear isomorphism $K\ot \overline{A}\rightarrow A$ and the assertion follows, since $H'$-colinear maps are exactly the $\mathbb{Z}$-graded $H$-colinear maps.
\end{proof}

\newcommand{\etalchar}[1]{$^{#1}$}
\providecommand{\bysame}{\leavevmode\hbox to3em{\hrulefill}\thinspace}
\providecommand{\MR}{\relax\ifhmode\unskip\space\fi MR }
\providecommand{\MRhref}[2]{%
  \href{http://www.ams.org/mathscinet-getitem?mr=#1}{#2}
}
\providecommand{\href}[2]{#2}

\end{document}